\documentclass[a4paper,12pt]{article}
\usepackage[english]{babel}
\usepackage{tikz}
\usetikzlibrary{matrix,arrows,decorations.markings}
\usepackage{amsmath,amsfonts,amssymb,amsthm,url,textcomp}
\usepackage{csquotes}
\usepackage{a4wide}
\usepackage[numbers]{natbib}
\usepackage{fancyhdr}
\usepackage{anyfontsize}
\usepackage{hyperref}
\usepackage{hhline}
\usepackage{leftidx}

\newtheorem{theorem}{Theorem}\numberwithin{theorem}{section}

\newtheorem{lemma}[theorem]{Lemma}

\newtheorem{proposition}[theorem]{Proposition}

\numberwithin{theoremm}{subsection}

\numberwithin{theoremmm}{subsubsection}

\theoremstyle{remark}

\newcommand{\Rad}{\operatorname{Rad}}
\newcommand{\Aut}{\operatorname{Aut}}

\newcommand{\cl}{\operatorname{cl}}

\renewcommand{\l}{\operatorname{l}}
\newcommand{\C}{\operatorname{C}}

\newcommand{\Soc}{\operatorname{Soc}}

\newcommand{\fix}{\operatorname{fix}}

\newcommand{\dl}{\operatorname{length}}

\newcommand{\cp}{\operatorname{cp}}
\newcommand{\Fit}{\operatorname{Fit}}

\begin{document}

\title{A note on finite groups with an automorphism inverting or squaring a non-negligible fraction of elements}

\author{Alexander Bors\thanks{University of Salzburg, Mathematics Department, Hellbrunner Stra{\ss}e 34, 5020 Salzburg, Austria. \newline E-mail: \href{mailto:alexander.bors@sbg.ac.at}{alexander.bors@sbg.ac.at} \newline The author is supported by the Austrian Science Fund (FWF):
Project F5504-N26, which is a part of the Special Research Program \enquote{Quasi-Monte Carlo Methods: Theory and Applications}. \newline 2010 \emph{Mathematics Subject Classification}: Primary: 20D25, 20D45, 20D60. Secondary: 05D05. \newline \emph{Key words and phrases:} Finite groups, Automorphisms, Powers of group elements, Commuting probability.}}

\date{\today}

\maketitle

\abstract{We show that for a finite group $G$, the commuting probability of $G$ can be explicitly bounded from below in a nontrivial way by a function in the maximum fraction of elements inverted resp. squared by an automorphism of $G$. Using these bounds together with a result of Guralnick and Robinson gives upper bounds on the index of the Fitting subgroup of $G$ under each of the two conditions that $G$ have an automorphism inverting resp. squaring at least $\rho|G|$ many elements in $G$, for $\rho\in\left(0,1\right]$ fixed. This is an improvement on previous results of the author.}

\section{Introduction}\label{sec1}

For an integer $e$ and a finite group $G$, denote by $\l_e(G)$ the maximum fraction of elements of $G$ mapped to their $e$-th power by a single automorphism of $G$. For $e=-1,2,3$, finite groups with sufficiently large $\l_e$-values are well-studied, and the gist of the results on them is that they are \enquote{close to being abelian} in some sense. In the preprint \cite{Bor16a}, the author studied finite groups whose $\l_e$-value for $e$ equal to one of the three numbers $-1,2$ or $3$ is bounded away from $0$. Denoting the solvable radical of $G$ by $\Rad(G)$ and the derived length of a solvable group $H$ by $\dl(H)$, the following was the main result of that preprint:

\begin{theorem}\label{oldTheo}
Let $\rho\in\left(0,1\right]$ be fixed, $G$ a finite group. Then:
\begin{enumerate}
\item If $G$ has an automorphism inverting at least $\rho|G|$ many elements of $G$, then both the index and the derived length of the solvable radical of $G$ are bounded in terms of $\rho$. More precisely, we then have $[G:\Rad(G)]\leq\rho^{-12.7650\ldots}$ and $\dl(\Rad(G))\leq\max(2,\log_{3/4}(2\rho)+3)$.
\item If $G$ has an automorphism squaring at least $\rho|G|$ many elements of $G$, then both the index and the derived length of the solvable radical of $G$ are bounded in terms of $\rho$. More precisely, we then have $[G:\Rad(G)]\leq\rho^{-4}$ and $\dl(\Rad(G))\leq 2\cdot\log_{3/4}(\rho)+1$.
\item If $G$ has an automorphism cubing at least $\rho|G|$ many elements of $G$, then the index of the solvable radical of $G$ is bounded in terms of $\rho$.
\end{enumerate}
\end{theorem}

We note that the method of proof of Theorem \ref{oldTheo}(3) actually gives an explicit upper bound on $[G:\Rad(G)]$ in terms of $\rho$, but that bound is not as simple as in the first two cases.

The aim of this note is to improve upon the results of Theorem \ref{oldTheo}(1,2). More precisely, we will prove the following, denoting the commuting probability of a finite group $G$ (i.e., the probability that two independently uniformly randomly chosen elements of $G$ commute) by $\cp(G)$, the Fitting subgroup of $G$ by $\Fit(G)$ and the nilpotency class of a finite nilpotent group $H$ by $\cl(H)$:

\begin{theorem}\label{newTheo}
Let $\rho\in\left(0,1\right]$ be fixed, $G$ a finite group. Then:
\begin{enumerate}
\item If $G$ has an automorphism inverting at least $\rho|G|$ many elements in $G$, then the following hold:
\begin{enumerate}
\item $\cp(G)\geq\frac{1}{12}\rho^5$,
\item $[G:\Fit(G)]\leq 144\rho^{-10}$,
\item $\dl(\Rad(G))\leq\max(2,\log_{3/4}(2\rho)+3)$,
\end{enumerate}
\item If $G$ has an automorphism squaring at least $\rho|G|$ many elements in $G$, then the following hold:
\begin{enumerate}
\item $\cp(G)\geq\rho^2$,
\item $[G:\Fit(G)]\leq\rho^{-4}$,
\item $\dl(\Rad(G))\leq \max(\{4\}\cup\{l\in\mathbb{Z}\mid l\geq0, 2^{l+1}\leq\frac{4l-7}{\rho^2}\})$,
\end{enumerate}
\end{enumerate}
\end{theorem}

We note that the main novelty in Theorem \ref{newTheo} are the lower bounds on $\cp(G)$; once they are established, the rest follows rather easily from results of \cite{GR06a} and \cite{Heg05a}. Furthermore, that $\l_2(G)\geq\rho$ implies the lower bound on $\cp(G)$ asserted in Theorem \ref{newTheo}(2,a) is a rather easy consequence of results from \cite{Bor16a}, so the only part of Theorem \ref{newTheo} for the proof of which we need an essentially new idea is subpoint (1,a). We will discuss this new idea in the next section. Finally, we note that similar results can be derived under the assumption $\l_3(G)\geq\rho$ if one assumes that the order of $G$ is odd, see Proposition \ref{oddOrderProp} below.

We will use the following notation throughout the paper: For a group $G$ and an element $g\in G$, $\C_G(g)$ denotes the centralizer of $g$ in $G$, $\zeta G$ the center of $G$, and $\tau_g:G\rightarrow G,x\mapsto gxg^{-1}$, denotes the conjugation by $g$ on $G$.

\section{Intersection of translates of the set of elements inverted by a finite group automorphism}

Our argument builds up on a part of a proof of the following well-known fact, which we review first:

\begin{proposition}\label{knownProp}
A finite group $G$ with $\l_{-1}(G)>\frac{3}{4}$ is abelian.
\end{proposition}

\begin{proof}[Proof (see \cite{GPa})]
Fix an automorphism $\alpha$ of $G$ inverting more than $\frac{3}{4}|G|$ many elements, and denote by $S$ the set of elements inverted by $\alpha$. For $s\in S$, since both $S$ and its translate $sS$ are subsets of $G$ size more than $\frac{3}{4}|G|$, it follows that $|sS\cap S|>\frac{1}{2}|G|$. Hence for more than $\frac{1}{2}|G|$ many $t\in S$, we have that $st\in S$ as well. It follows that $t^{-1}s^{-1}=(st)^{-1}=\alpha(st)=\alpha(s)\alpha(t)=s^{-1}t^{-1}$, or equivalently $t\in\C_G(s)$. Therefore, $|\C_G(s)|>\frac{1}{2}|G|$, and thus $\C_G(s)=G$, i.e., $s\in\zeta G$, by Lagrange's theorem. We just showed that $S\subseteq\zeta G$, whence $\zeta G=G$ by another application of Lagrange's theorem, and so $G$ is abelian.
\end{proof}

The gist of this argument is that because $S$ is so large, the intersection of $S$ with the translate $sS$ by any element $s\in S$ is also large (first inference), and therefore, all $s\in S$ have large centralizers (second inference). Both inferences have analogues under the weaker assumption that $|S|\geq\rho|G|$ for some fixed $\rho\in\left(0,1\right]$. The following elementary lemma on intersections of \enquote{non-negligible} subsets of finite sets generalizes the first inference:

\begin{lemma}\label{combinatorialLem}
Let $\rho\in\left(0,1\right]$, $M$ a finite set, $(S_i)_{i\in I}$ a nonempty family of subsets of $M$ such that $|S_i|\geq\rho|M|$ for all $i\in I$. Set $k(\rho):=\lceil\rho^{-1}\rceil+1$ (so that $k(\rho)\cdot\rho\geq 1+\rho$) and $t(\rho):=\frac{\rho}{\Delta_{k(\rho)-1}}=\frac{\rho}{\Delta_{\lceil\rho^{-1}\rceil}}$, where $\Delta_n:=\frac{1}{2}n(n+1)$ denotes the $n$-th triangle number. Then the following hold:
\begin{enumerate}
\item If $J\subseteq I$ with $|J|\geq k(\rho)$, then there exist distinct $i,j\in I$ such that $|S_i\cap S_j|\geq t(\rho)|M|$.
\item There exists $i\in I$ such that for at least $\frac{|I|-(k(\rho)-1)}{k(\rho)-1}$ many $j\in I\setminus\{i\}$, we have $|S_i\cap S_j|\geq t(\rho)|M|$.
\item If $|I|\geq 2(k(\rho)-1)$, then there exists $i\in I$ such that for at least $\frac{1}{2(k(\rho)-1)}|I|$ many $j\in I\setminus\{i\}$, we have $|S_i\cap S_j|\geq t(\rho)|M|$.
\end{enumerate}
\end{lemma}

\begin{proof}
For (1): We may of course assume w.l.o.g.~that $|J|=k(\rho)$, and show the assertion for such $J$ by contradiction; assume that $|S_i\cap S_j|<t(\rho)|M|$ for all distinct $i,j\in J$. Say $J=\{j_1,\ldots,j_{k(\rho)}\}$, and set, for $l=1,\ldots,k(\rho)$, $U_l:=\bigcup_{i=1}^{l}{S_{j_i}}$. We show by induction on $l$ that

\begin{equation}\label{eq1}
|U_l|>(l\cdot\rho-\Delta_{l-1}t(\rho))|M|
\end{equation}

for $l=2,\ldots,k(\rho)$. Indeed, we find that

\[
|U_2|=|S_{j_1}\cup S_{j_2}|\geq|S_{j_1}|+|S_{j_2}|-|S_{j_1}\cap S_{j_2}|>\rho|M|+\rho|M|-t(\rho)|M|=(2\rho-t(\rho))|M|,
\]

and if the assertion has been verified up to $l-1$, it follows that

\begin{align*}
|U_l| &=|U_{l-1}\cup S_{j_l}|\geq|U_{l-1}|+|S_{j_l}|-|U_{l-1}\cap S_{j_l}| \\
&\geq((l-1)\rho-\Delta_{l-2}t(\rho))|M|+\rho|M|-|\bigcup_{i=1}^{l-1}{S_{j_i}\cap S_{j_l}}| \\
&>(l\rho-\Delta_{l-2}t(\rho))|M|-(l-1)t(\rho)|M|=(l\rho-\Delta_{l-1}t(\rho))|M|,
\end{align*}

as required. However, by setting $l:=k(\rho)$ in Equation (\ref{eq1}), we get that

\[
|U_{k(\rho)}|>(k(\rho)\rho-\Delta_{k(\rho)-1}t(\rho))|M|\geq(1+\rho-\rho)|M|=|M|,
\]

a contradiction.

For (2): If $|I|\leq k(\rho)-1$, there is nothing to show, so assume that $|I|\geq k(\rho)$. Let $J\subseteq I$ be maximal such that for all distinct $i,j\in J$, we have $|S_i\cap S_j|<t(\rho)|M|$. By (1), $|J|\leq k(\rho)-1$. Set $K:=I\setminus J$; then $|K|\geq |I|-(k(\rho)-1)$. Furthermore, by maximality of $J$, there exists a function $\iota:K\rightarrow J$ such that for all $j\in K$, $|S_{\iota(j)}\cap S_j|\geq t(\rho)|M|$. For at least one $i\in J$, the fiber $\iota^{-1}[\{i\}]$ has size at least $\frac{|K|}{|J|}\geq\frac{|I|-(k(\rho)-1)}{k(\rho)-1}$, and any such $i$ \enquote{does the job}.

For (3): This follows from (2), since by assumption,

\[
\frac{|I|-(k(\rho)-1)}{k(\rho)-1}=\frac{|I|}{k(\rho)-1}-1\geq\frac{|I|}{k(\rho)-1}-\frac{1}{2}\frac{|I|}{k(\rho)-1}=\frac{1}{2(k(\rho)-1)}|I|.
\]
\end{proof}

The second inference has the following generalization:

\begin{lemma}\label{translateLem}
Let $\epsilon\in\left(0,1\right]$, $G$ a finite group, $\alpha$ an automorphism of $G$, $S$ the set of elements of $G$ inverted by $\alpha$. Assume that $s,t\in S$ are such that $|sS\cap tS|\geq\epsilon|G|$. Then $|\C_G(st^{-1})|\geq\epsilon|G|$.
\end{lemma}

\begin{proof}
By assumption, we have $|S\cap s^{-1}tS|=|s^{-1}(sS\cap tS)|=|sS\cap tS|\geq\epsilon|G|$. In other words, for at least $\epsilon|G|$ many $u\in S$, we have that $s^{-1}tu\in S$ as well. It follows that $u^{-1}t^{-1}s=(s^{-1}tu)^{-1}=\alpha(s^{-1}tu)=\alpha(s)^{-1}\alpha(t)\alpha(u)=st^{-1}u^{-1}$, or equivalently $\tau_{s^{-1}}(st^{-1})=t^{-1}s=\tau_u(st^{-1})$, whence for all such $u$, we have $su\in\C_G(st^{-1})$, and the assertion follows.
\end{proof}

\section{Proof of Theorem \ref{newTheo}}

For (1,a): First, assume that $|G|<2(k(\rho)-1)\rho^{-1}=2\lceil\rho^{-1}\rceil\rho^{-1}\leq 4\rho^{-2}$. Then if we had $\cp(G)<\frac{1}{12}\rho^5$, we would get the contradictory chain of inequalities $\frac{1}{12}\rho^5>\cp(G)\geq |G|^{-1}>\frac{1}{4}\rho^2$. Therefore, we may assume that $|G|\geq 2(k(\rho)-1)\rho^{-1}$. Let $\alpha$ be an automorphism of $G$ inverting at least $\rho|G|$ many elements of $G$, and let $S$ be the set of such elements. Note that by assumption, $|S|\geq\rho|G|\geq 2(k(\rho)-1)$. Hence by applying Lemma \ref{combinatorialLem}(3) to the family $(sS)_{s\in S}$ of subsets of $G$, we get that there exists $s\in S$ such that for at least $\frac{|S|}{2(k(\rho)-1)}\geq\frac{\rho}{2(k(\rho)-1}|G|$ many elements $t\in S$, $|sS\cap tS|\geq t(\rho)|G|$. By Lemma \ref{translateLem}, this yields that for all such $t$, $|\C_G(st^{-1})|\geq t(\rho)|G|$. Hence

\begin{align*}
\cp(G) &\geq\frac{\rho}{2(k(\rho)-1)}\cdot t(\rho)=\frac{\rho}{2\lceil\rho^{-1}\rceil}\cdot\frac{\rho}{\Delta_{\lceil\rho^{-1}\rceil}}=\frac{\rho^2}{\lceil\rho^{-1}\rceil^2(\lceil\rho^{-1}\rceil+1)} \\
&\geq\frac{\rho^2}{(\rho^{-1}+1)^2(\rho^{-1}+2)}\geq\frac{\rho^2}{(2\rho^{-1})^2\cdot 3\rho^{-1}}=\frac{1}{12}\rho^5.
\end{align*}

For (1,b): This follows immediately from (1,a) and $\cp(G)\leq[G:\Fit(G)]^{-1/2}$, see \cite[Theorem 10(ii)]{GR06a}.

For (1,c): This is part of the statement of \cite[Theorem 1.1.3(1)]{Bor16a}.

For (2,a): Fix an automorphism $\alpha$ squaring at least $\rho|G|$ many elements of $G$, and let $S$ be the set of such elements. By \cite[Lemma 2.1.6]{Bor16a}, this implies that $\alpha$ has at most $\rho^{-1}$ many fixed points, and thus, by \cite{Gus73a} and \cite[Lemma 2.1.2]{Bor16a}, we have $\rho\leq\frac{|S|}{|G|}\leq\cp(G)\cdot\rho^{-1}$, whence $\cp(G)\geq\rho^2$, as required.

For (2,b): This follows from (2,a) just like (1,b) follows from (1,a).

For (2,c): By (2,a) and \cite[Lemma 2(iii)]{GR06a}, we get that $\rho^2\leq\cp(\Rad(G))$, which implies the assertion via \cite[Theorem 12(i)]{GR06a}.\qed

\section{Concluding remarks}

\subsection{On the use of the CFSG for our results}

By showing that $\cp(G)$ can be bounded from below in terms of both $\l_{-1}(G)$ and $\l_2(G)$, we could reduce bounding other parameters of $G$ (such as the index of the Fitting subgroup) in terms of both $\l_{-1}(G)$ and $\l_2(G)$ to Guralnick and Robinson's results on the commuting probability from \cite{GR06a}. Our arguments leading to the lower bounds of the form $\cp(G)\geq f_1(\l_{-1}(G))$ and $\cp(G)\geq f_2(\l_2(G))$ are elementary; they do not require the CFSG nor any other tools from outside elementary group theory, such as character theory.

However, we note that Guralnick and Robinson's result $\cp(G)\geq\rho \Rightarrow [G:\Fit(G)]\leq\rho^{-2}$ \cite[Theorem 10(ii)]{GR06a}, which we used to get the simple bounds on $[G:\Fit(G)]$ from Theorem \ref{newTheo}(1,b and 2,b), does require the CFSG. More precisely, \cite[Theorem 10(ii)]{GR06a} depends on two other results from the same paper:

\begin{itemize}
\item \cite[Theorem 4(ii)]{GR06a}, stating that in a finite \emph{solvable} group $G$, we have $\cp(G)\leq\cp(\Fit(G))^{1/2}[G:\Fit(G)]^{-1/2}$, and
\item \cite[Theorem 9]{GR06a}, which says that $\cp(G)\leq[G:\Rad(G)]^{-1/2}$ in \emph{all} finite groups.
\end{itemize}

The proof of \cite[Theorem 4(ii)]{GR06a} does not require the CFSG (though it does require quite a bit of character theory, more precisely one of the main results of \cite{Kno84a}), but the CFSG is used for \cite[Theorem 9]{GR06a}. However, just to show CFSG-freely that $\cp(G)\geq\rho$ implies that $[G:\Fit(G)]$ is bounded \emph{per se} (without the explicit bound established with the CFSG), it would suffice to show CFSG-freely that $\cp(G)\geq\rho$ implies that $[G:\Rad(G)]$ is bounded in terms of $\rho$ (and combine this with the CFSG-free \cite[Theorem 4(ii)]{GR06a} just as Guralnick and Robinson did). And this is indeed possible:

\begin{proposition}\label{cfsgFreeProp}(CFSG-free)
For finite groups $G$, $\cp(G)\to0$ as $[G:\Rad(G)]\to\infty$.
\end{proposition}

\begin{proof}
Fix $\rho\in\left(0,1\right]$, and assume that $G$ is a finite group with $\cp(G)\geq\rho$. We will show that $[G:\Rad(G)]$ is bounded. By \cite[Lemma 2(iv)]{GR06a} and the fact that $\cp(G)\leq\frac{5}{8}$ when $G$ is nonabelian \cite{Gus73a}, we get that the number of non-abelian composition factors of $G$, counting with repetitions, is bounded. Furthermore, the order of each such composition factor $S$ is also bounded, in view of $\cp(S)\geq\rho$ (which follows from \cite[Lemma 2(iii)]{GR06a}). This is because by simplicity of $S$, the minimum index of a proper subgroup of $S$ is bounded from below by the smallest positive integer $r(S)$ such that $r(S)!\geq |S|$, and $r(S)\to\infty$ as $|S|\to\infty$. Hence $\cp(S)\leq\frac{1-1/|S|}{r(S)}+\frac{1}{|S|}\to 0$ as $|S|\to\infty$, because centralizers of nontrivial elements of $S$ are proper subgroups.

We now use some facts explained in detail in \cite[pp.~88ff.]{Rob96a}. Since $G/\Rad(G)$ has trivial solvable radical, its socle is a direct product of nonabelian finite simple groups, all of which are composition factors of $G$. Hence in view of the last paragraph, $|\Soc(G/\Rad(G))|$ is bounded, and thus $|G/\Rad(G)|=[G:\Rad(G)]$ is bounded, since $G/\Rad(G)$ embeds into $\Aut(\Soc(G/\Rad(G)))$.
\end{proof}

\subsection{Bounding \texorpdfstring{$\cp(G)$}{cp(G)} in terms of \texorpdfstring{$\l_3(G)$}{l3(G)}?}

Note that while the author was able to show in \cite{Bor16a} that under an assumption of the form $\l_3(G)\geq\rho$, the index $[G:\Rad(G)]$ is bounded in terms of $\rho$, it is still open whether this condition is also strong enough to imply that the derived length of $\Rad(G)$ is bounded. Of course, if one could bound $\cp(G)$ from below in terms of $\l_3(G)$ (as we did for $\l_{-1}(G)$ and $\l_2(G)$ here), this would solve the problem instantly. We note the following argument, which covers at least the groups $G$ of odd order:

\begin{proposition}\label{oddOrderProp}
Let $G$ be a finite group of odd order. Then $\cp(G)\geq\l_3(G)^2$.
\end{proposition}

\begin{proof}
Set $\rho:=\l_3(G)$, fix an automorphism $\alpha$ of $G$ cubing $\rho|G|$ many elements of $G$, and let $S$ be the set of such elements. That is, we have

\[
\rho|G|=|S|=|\{g\in G\mid\alpha(g)=g^3\}|=|\{g\in G\mid g^{-1}\alpha(g)=g^2\}|\leq[G:\fix(\alpha)],
\]

where the last inequality holds since the map $g\mapsto g^{-1}\alpha(g)$ is constant on right cosets of the subgroup $\fix(\alpha)$ consisting of the fixed points of $\alpha$, whereas the map $g\mapsto g^2$ is injective on $G$. Hence $|\fix(\alpha)|\leq\rho^{-1}$, and we can conclude as in the proof of Theorem \ref{newTheo}(2,a).
\end{proof}

\end{document}